\documentclass[11pt,bezier]{article}
\usepackage{amsmath,amssymb,amsfonts,euscript,graphicx}

\newtheorem{prethm}{{\bf Theorem}}

\newenvironment{thm}{\begin{prethm}{\hspace{-0.5
               em}{\bf.}}}{\end{prethm}}

\newtheorem{prepro}[prethm]{{\bf Theorem}}

\newtheorem{preprop}[prethm]{{\bf Proposition}}

\newtheorem{precor}[prethm]{{\bf Corollary}}

\newenvironment{cor}{\begin{precor}{\hspace{-0.5
               em}{\bf.}}}{\end{precor}}

\newtheorem{preconj}[prethm]{{\bf Conjecture}}

\newtheorem{preremark}[prethm]{{\bf Remark}}

\newenvironment{remark}{\begin{preremark}\rm{\hspace{-0.5
               em}{\bf.}}}{\end{preremark}}

\newtheorem{preexample}[prethm]{{\bf Example}}

\newtheorem{prelem}[prethm]{{\bf Lemma}}

\newenvironment{lem}{\begin{prelem}{\hspace{-0.5
               em}{\bf.}}}{\end{prelem}}

\newtheorem{prelam}{{\bf Lemma}}

\newtheorem{preproof}{{\bf Proof.}}

\newenvironment{proof}[1]{\begin{preproof}{\rm
               #1}\hfill{$\Box$}}{\end{preproof}}

\title{\bf \large When the Annihilator Graph of a Commutative Ring \\Is Planar or Toroidal?  }
\author{{\normalsize  {\sc M.J. Nikmehr${}^{\mathsf{a}}$, {\sc R. Nikandish${}^{\mathsf{b}}$} and {\sc M. Bakhtyiari${}^{\mathsf{a}}$}  }
}\vspace{3mm}\\
{\footnotesize{${}^{\mathsf{a}}$\it Faculty of Mathematics, K.N. Toosi
University of Technology, }}\\
{\footnotesize{\rm P.O. BOX \rm{16315-1618}, Tehran, Iran}}\\
{\footnotesize{ $\mathsf{nikmehr@kntu.ac.ir}$}}\quad\quad
{\footnotesize{$\mathsf{m.bakhtyiari55@gmail.com}$}}\\
{\footnotesize{${}^{\mathsf{b}}$\it Department of Mathematics, Jundi-Shapur University of Technology,}}\\
{\footnotesize{\rm P.O. BOX \rm{64615-334},
Dezful, Iran}}\\
{\footnotesize{ $\mathsf{r.nikandish@jsu.ac.ir}$}}\\
{\footnotesize{$\mathsf{}$ }}}

\date{}

\begin{document}

\maketitle

\begin{abstract}
{\small Let $R$ be a commutative ring with identity, and let $Z(R)$ be the set of zero-divisors of $R$. The  annihilator graph of $R$ is defined as the undirected graph $AG(R)$ with the vertex set $Z(R)^*=Z(R)\setminus\{0\}$, and two distinct vertices $x$ and $y$ are adjacent if and only if $ann_R(xy)\neq ann_R(x)\cup ann_R(y)$. In this paper, all rings whose annihilator graphs can be embed on the plane or torus are classified.  }
\end{abstract}
\section{Introduction}

Recently, a major part of research in algebraic combinatorics has been devoted to the application of graph theory and
combinatorics in abstract algebra. There are a lot of papers which apply combinatorial
methods to obtain algebraic results in ring theory (see \cite{nikak}, \cite{nikandish}, \cite{Badawi} and \cite{Heydari}). Moreover, for most recent
study in this field see \cite{ab} and \cite{su}.
\par
Throughout this paper $R$ is a commutative ring with identity which is not an integral domain. We denote by  $\mathrm{Min}(R)$, $\mathrm{Nil}(R)$ and $\mathrm{U}(R)$, the set of all minimal prime ideals of $R$, the set of all nilpotent elements of $R$ and the set of all invertible elements of $R$, respectively.
Also, the set of all zero-divisors of an $R$-module $M$, which is denoted by $Z(M)$, is the set $$Z(M)=\{r\in R \, |\,rx=0, \rm{\,for\,\, some\,\, nonzero\,\, element}\,\, {\it x} \,\,\rm{ in}\,\, {\it M}\}.$$
A finite field of order $n$ is denoted by $\mathbb{F}_n$. By $\mathrm{dim}(R)$ and $\mathrm{depth}(R)$, we mean the dimension and depth of $R$, see \cite{sha}.  For every ideal $I$ of $R$, we denote the  annihilator of $I$ by $\mathrm{Ann}(I)$. For a subset $A$ of a ring $R$ we let $A^*=A\setminus\{0\}$.  The ring $R$ is said to be \textit{reduced} if it has no non-zero
nilpotent element. Let $R$ be a Noetherian local ring. Then $R$ is said to be a \textit{Cohen-Macaulay} ring if $\mathrm{depht}(R)=\mathrm{dim}(R)$. In general, if $R$ is a Noetherian ring, then $R$ is a Cohen-Macaulay ring if $R_\mathfrak{m}$ is a Cohen-Macaulay ring, for all maximal ideals $\mathfrak{m}$, where $R_\mathfrak{m}$ is the localization of $R$ at $\mathfrak{m}$. Also, a Noetherian local ring $R$ is called \textit{Gorenstein} if $R$ is Cohen-Macaulay and $\mathrm{dim}_{R/\mathfrak{m}}(\mathrm{soc}(R))=1$, where $\mathfrak{m}$ is the unique maximal ideal of $R$. In general, if $R$ is a Noetherian ring, then $R$ is a Gorenstein ring if $R_\mathfrak{m}$ is a Gorenstein ring, for all maximal ideals $\mathfrak{m}$. For any undefined notation or terminology in ring theory, we refer the reader to \cite{her, sha}.
\par
 Let $G=(V,E)$ be a graph, where $V=V(G)$ is the set of vertices and $E=E(G)$ is the set of edges. By $K_n$ and $K_{m,n}$, we mean the complete graph of order $n$ and the complete bipartite graph with part sizes $m$ and $n$, respectively. Moreover, by $\overline{G}$ we denote the complement of $G$. The graph $H=(V_0,E_0)$ is a subgraph of $G$ if $V_0\subseteq V$ and $E_0 \subseteq E$. Moreover, $H$ is called an \textit{induced subgraph by} $V_0$,  denoted by $G[V_0]$, if $V_0\subseteq V$ and $E_0=\{\{u,v\}\in E\, |\,u,v\in V_0\}$. Let $G_1$ and $G_2$ be two graphs.  \textit{The subdivision of a graph} $G$ is a graph obtained from $G$ by subdividing some of the edges, that is, by replacing the edges by paths having at most their endvertices in common. By $G_1\vee G_2$ and $G_1=G_2$, we mean the $\it join$ of $G_1$, $G_2$ and $G_1$ is identical to $G_2$, respectively. Let $S_k$ denote the sphere with $k$ handles, where $k$ is a non-negative integer, that is, $S_k$ is an oriented surface of genus $k$. The \textit {genus} of a graph $G$, denoted $\gamma(G)$, is the minimal integer $n$ such that the graph can be embedded in $S_n$ (see \cite[Chapter 6]{west}). Intuitively, $G$ is embedded in a surface if it can be drawn in the surface so that its edges intersect only at their common vertices. A genus $0$ graph is called a \textit{planar graph} and a genus $1$ graph is called a \textit{toroidal graph}.  It is well known that $$\gamma(K_n)={\lceil \dfrac{(n-3)(n-4)}{12}\rceil}  \,\,\,\,\,\,\,\, \rm {if}\,\,\, \textit{n}\geq 3 \,\,\, \rm{and}$$
 $$\gamma(K_{m,n})={\lceil \dfrac{(m-2)(n-2)}{4}\rceil}  \,\,\,\,\,\,\,\, \rm {if}\,\,\, \textit{n,m}\geq 2 \,\,\, $$
 For any undefined notation or terminology in ring theory, we refer the reader to \cite{west}.
 \par
 \noindent The  \textit{annihilator graph} of a ring $R$ is defined as the graph $AG(R)$ with the vertex set $Z(R)^*=Z(R)\setminus\{0\}$, and two distinct vertices $x$ and $y$ are adjacent if and only if $ann_R(xy)\neq ann_R(x)\cup ann_R(y)$. This graph was first introduced and investigated in \cite{Badawi} and many of interesting properties of an annihilator graph were studied. This paper is devoted to classify all rings whose annihilator graphs are planar or toroidal.
{\section{Planar Annihilator  Graphs}\vspace{-2mm}

In this section, we characterize all rings whose annihilator graphs are planar. Moreover, it is shown that the genus of the annihilator graph associated with an infinite ring is either zero or infinite. First, we recall a series of necessary results.
\begin{lem}\label{lemma1} {\rm \cite[Lemma 2.1]{ColorBadawi}}
 Let $R$ be a ring and $x,y$ be distinct elements  of $Z(R)^*$. Then the following statements are equivalent.

 $(1)$ $x-y$ is an edge of $AG(R)$.

 $(2)$ $Rx\cap ann_R(y)\neq(0)$ and $Ry\cap ann_R(x)\neq(0)$.

 $(3)$ $x\in Z(Ry)$ and $y\in Z(Rx)$.
\end{lem}
\begin{lem}\label{lemma2} {\rm \cite[Lemma 2.2]{ColorBadawi}}
 Let $R$ be a ring.

 $(1)$ Let $x, y$ be elements of $Z(R)^*$. If $ann_R(x)\nsubseteq ann_R(y)$ and $ann_R(y)\nsubseteq ann_R(x)$, then  $x-y$ is an edge of $AG(R)$. Moreover, if $R$ is a reduced ring, then the converse is also true.

 $(2)$ Let $R\cong R_1\times \cdots\times R_n$, $x= (x_1,\dots, x_n)$ and $y= (y_1,\dots,y_n)$, where $n$ is a positive integer, every $R_i$ is a ring and $x_i,y_i\in R_i$, for every $1\leq i\leq n$. If $R_ix_i\cap ann_{R_i}(y_i)\neq(0)$ and $R_jy_j\cap ann_{R_j}(x_j)\neq (0)$, for some $1\leq i,j\leq n$, then  $x-y$ is an edge of $AG(R)$. In particular, if $x_i-y_i$ is an edge of $AG(R_i)$ or $x_i=y_i\in\mathrm{Nil}(R_i)^* $, for some $1\leq i\leq n$, then  $x-y$ is an edge of $AG(R)$.
 \end{lem}

\begin{lem}\label{lemmae}
 Let $R$ be a reduced ring and contains a minimal ideal. Then $R$ is decomposable.
\end{lem}
\begin{proof}
 {The proof is obtained by \cite[2.7]{Rob}. }
\end{proof}
\par To classify planar annihilator graphs, we need a celebrated theorem due to Kuratowski.

 \begin{thm} {\rm (\cite[Theorem 6.2.2]{west}}
 A graph is planar if and only if it contains
no subdivision of either $K_{3,3}$ or $K_5$.
 \end{thm}

\begin{thm}\label{chit4=ome}
Let $R$ be a ring such that  $R\cong R_1\times \cdots  \times R_n$, where $n$ is a positive integer and $R_i$ is a ring, for every $1\leq i \leq n$. Then the following statements hold.

$(1)$ If  $n\geq 4$, then $AG(R)$ is not planar.

$(2)$ If  $n=3$ and $AG(R)$ is planar, then $R\cong \mathbb{Z}_2\times \mathbb{Z}_2\times\mathbb{Z}_2$.
\end{thm}
\begin{proof}
  {$(1)$ We need only to show that  $AG(R)$ is not planar for $n=4$. Since the set $\{(1,1,0,0),(0,1,1,0),(0,0,1,1),(1,0,1,0),(0,1,0,1)\}$ is a complete subgraph of $AG(R)$, $K_5$ is a subgraph of $AG(R)$. The result now follows from Kuratowski's Theorem.

  $(2)$ Let $R\cong R_1\times R_2\times R_3$. Assume to the contrary and without loss of generality,  $R_1\neq\mathbb{Z}_2$. Let $x\in R_1\setminus \{0,1\}$. Then it is not hard to check that the vertices of the set $\{(1,0,1),(x,0,0),(x,0,1)\}$, the vertices of the set $\{(0,1,1),(1,1,0),(0,1,0)\}$ together with the path $(x,0,0)-(0,0,1)-(1,1,0)$ forms a subgraph that  contains a subdivision of $K_{3,3}$, a contradiction. So $R\cong \mathbb{Z}_2\times \mathbb{Z}_2\times\mathbb{Z}_2$. }
\end{proof}
\par In the next theorem, we characterize reduced rings whose annihilator graphs are planar.
 \begin{thm}\label{comp5l}
  Let $R$ be a reduced ring. Then $AG(R)$ is planar if and only if one of the following statements holds.

  $(1)$ $R\cong \mathbb{Z}_2\times \mathbb{Z}_2\times\mathbb{Z}_2$.

  $(2)$ $|\mathrm{Min}(R)|=2$ and one of the minimal prime ideals of $R$ has at most three distinct elements.
\end{thm}
\begin{proof}
{ Suppose that  $AG(R)$ is planar and let $x\in Z(R)^*$. Since $R$ is a reduced ring, we have  $Rx\cap ann_R(x)=(0)$. If $|Rx|=|{ann_R}(x)|=\infty$, then obviously  $AG(R)$ is not planar, a contradiction. If either $|Rx|$ or $|{ann_R}(x)|$ is finite, then $R$ has a minimal ideal and so  by  Lemma \ref{lemmae}, $R$ is decomposable. Assume that $R\cong R_1\times R_2$, where $R_1,R_2$ are two rings. If $|\mathrm{Min}(R)|=2$, then by \cite[Theorem 3.7]{Badawi}, one of the minimal prime ideals of $R$ has at most three distinct elements. If $|\mathrm{Min}(R)|\geq3$, without loss of generality, we may assume that $|\mathrm{Min}(R_2)|\geq2$. Thus $Z(R_2)\neq (0)$. By repeating the above argument we conclude that $R_2$ is decomposable.  Therefore, one may assume that $R\cong R_1\times R_2\times R_3$, where  $R_1,R_2, R_3$ are three rings. By part (2) of Theorem \ref{chit4=ome}, $R\cong \mathbb{Z}_2\times \mathbb{Z}_2\times\mathbb{Z}_2$.

Conversely, if $R\cong \mathbb{Z}_2\times \mathbb{Z}_2\times\mathbb{Z}_2$, then one may easily see that $AG(R)$ is planar. Also, if $|\mathrm{Min}(R)|=2$ and one of the minimal prime ideals of $R$ has at most three distinct elements, then the result follows from \cite[Theorem 3.7]{Badawi}.  }
\end{proof}

\par To characterize non-reduced rings whose annihilator graphs are planar we state the following lemmas.
\begin{lem}\label{lemma43}
  {\rm \cite[Lemma 2.2]{khashyarmanesh}} Let $R$ be a ring and $\mathfrak{m}$ be a maximal ideal in $R$. If $Ann(\mathfrak{m})\neq 0$, then $\mathfrak{m}=Z(Ann(\mathfrak{m}))$.
\end{lem}
\begin{lem}\label{lemma439}
 Let $R$ be a ring and $\mathfrak{m}_1,\mathfrak{m}_2$ be two maximal ideals of $R$ such that $Ann(\mathfrak{m}_1)\neq (0), Ann(\mathfrak{m}_2)\neq (0)$. Then $K_{|{\mathfrak{m}_1\setminus\mathfrak{m}_2}|,{|\mathfrak{m}_2\setminus\mathfrak{m}_1}|}$ is a subgraph of $AG(R)$.
\end{lem}
\begin{proof}
 { Let $x\in\mathfrak{m}_1\setminus\mathfrak{m}_2$ and $y\in\mathfrak{m}_2\setminus\mathfrak{m}_1 $. We claim that $Ann(\mathfrak{m}_2) \cap ann_R(x)=0$. Assume to the contrary, there exists an element $z\in Ann(\mathfrak{m}_2) \cap ann_R(x)$. Hence $zx=0$. Now, Lemma \ref{lemma43} implies that $x\in \mathfrak{m}_2$, a contradiction. Similarly, $Ann(\mathfrak{m}_1) \cap ann_R(y)=0$. Since  $\mathfrak{m}_1+\mathfrak{m}_2=R$,  $Ann(\mathfrak{m}_1)\neq Ann(\mathfrak{m}_2)$. Hence $Ann(\mathfrak{m}_1)\subseteq ann_R(x)\nsubseteq ann_R(y), \,Ann(\mathfrak{m}_2)\subseteq ann_R(y)\nsubseteq ann_R(x) $ and so by part (2) of Lemma \ref{lemma2}, $x-y$ is an edge of $AG(R)$. }
\end{proof}
\begin{thm}\label{comp95l}
  Let $R$ be a non-reduced ring. Then $AG(R)$ is planar if and only if one of the following statements holds:

 $(1)$ $R$ is ring-isomorphic to either $\mathbb{Z}_2\times\mathbb{Z}_4$ or $\mathbb{Z}_2\times\mathbb{Z}_2[X]/(X^2)$.

 $(2)$ $Ann(Z(R)) $ is a prime ideal of $R$ and $2\leq|\mathrm{Nil}(R)|\leq3$.

 $(3)$ $Z(R)=\mathrm{Nil}(R) $ and $4\leq|\mathrm{Nil}(R)|\leq5$.
\end{thm}
\begin{proof}
{  Suppose that $AG(R)$ is planar. We consider following cases.

\textbf{Case 1.} $R$ is decomposable. Let $R\cong R_1\times R_2$, where $R_1,R_2$ are two rings. One may assume that there exists a non-zero element $a \in \mathrm{Nil}(R_1)$. We show that  $|Z(R_1)|=2$. If $|Z(R_1)| \geq 3$, then by part (2) of Lemma \ref{lemma2}, the vertices contained in the set $\{(1,0),(u,0),(a,0)\}$ and  the vertices contained in the set $\{(0,1),(x,1),(a,1) \}$ form $K_{3,3}$, where $1\neq u \in U(R_1)$ and $x$ is a neighbor of $a$ in $AG(R_1)$, a contradiction (note that $|\mathrm{Nil}(R_1)|\leq |U(R_1)|$). This implies that $|Z(R_1)|=2$. Similarly, if $x\in R_2\setminus \{0,1\}$, then the vertices of the set $\{(1,0),(u,0),(a,0)\}$ and the vertices of the set $\{(0,1),(a,x),(a,1)\}$ form $K_{3,3}$, a contradiction. So $R$ is ring-isomorphic to either $\mathbb{Z}_2\times\mathbb{Z}_4$ or $\mathbb{Z}_2\times\mathbb{Z}_2[X]/(X^2)$.

\textbf{Case 2.} $R$ is indecomposable. By \cite[Theorem 3.10]{Badawi},  $2\leq|\mathrm{Nil}(R)|\leq5$. Then either $2\leq|\mathrm{Nil}(R)|\leq3$ or $4\leq|\mathrm{Nil}(R)|\leq5$. First assume that $Z(R)= \mathrm{Nil}(R)$. If $4\leq|\mathrm{Nil}(R)|\leq5$, then $(3)$ holds. If  $2\leq|\mathrm{Nil}(R)|\leq3$, then $\mathrm{Nil}(R)^2=(0)$ and since  $Z(R)= \mathrm{Nil}(R)$, $Ann(Z(R)) $ is a prime ideal of $R$ and so $(2)$ holds. Now, let $Z(R)\neq \mathrm{Nil}(R)$ and $Ra$ be a minimal ideal, for some $a\in\mathrm{Nil}(R)^*$. Since  $R$ is indecomposable and $Z(R)\neq \mathrm{Nil}(R)$, we conclude that $|ann_R(a)|$ has infinitely many elements. If $xy=0$,  for some $x,y\in Z(R)\setminus\mathrm{Nil}(R)$, then the vertices of the set $\{x,x^2,x^3\}$ and  the vertices of the set $\{y,y^2,y^3\}$ are adjacent, a contradiction (as $R$ is indecomposable). So $ann_R(x)\subseteq\mathrm{Nil}(R)$, for every $x\in Z(R)\setminus\mathrm{Nil}(R) $. Now, let $a\neq b\in\mathrm{Nil}(R)^*$. We claim that $b$ is adjacent to all vertices contained in $ann_R(a)$. To see this, we consider two subcases.

\textbf{Subcase 1.} $Ra\subseteq Rb$. Let $x$ be an arbitrary element of $ ann_R(a)\setminus\mathrm{Nil}(R)$. If $xb=0$, then there is nothing to prove. So let $xb\neq 0$ and $xb^{n-1}\neq0$, $xb^n=0$, for a positive integer $n$. Thus $xb^{n-1}\in Rx\cap ann_R(b)$. Since $Ra\subseteq Rb$, we deduce that $ Rb\cap ann_R(x)\neq (0)$. Now, by Lemma \ref{lemma1}, $x-b$ is an edge of $AG(R)$.

\textbf{Subcase 2.} $Ra\nsubseteq Rb$. Since $Ra$ is a minimal ideal,  $Ra\cap Rb=(0)$. So $Rb$ contains a minimal ideal, say $Rc$, for some  $c\in\mathrm{Nil}(R)$. Thus $ann_R(c)$ is a maximal ideal of $R$. If $ann_R(a)\neq ann_R(c)$, then by  Lemma \ref{lemma439}, we get  a contradiction (as $ann_R(a)$ is a maximal ideal, too). Thus $ann_R(a)= ann_R(c)$. The fact $Rc\subseteq Rb$ together with subcase 1  imply that $b$ is adjacent to all vertices contained in $ann_R(a)$.

So the claim is proved. This together with the planarity of $AG(R)$ imply that $2\leq|\mathrm{Nil}(R)|\leq3$ and hence $\mathrm{Nil}(R)$ is a minimal ideal. Since  $ann_R(x)\subseteq\mathrm{Nil}(R)$ for every $x\in Z(R)\setminus\mathrm{Nil}(R) $, we have $Ann(Z(R))= \mathrm{Nil}(R)$ and $\mathrm{Nil}(R)$ is a prime ideal of $R$.

Conversely, if either $(1$) or $(2)$ is hold, then obviously $AG(R)$ is planar. Moreover if  $Ann(Z(R)) $ is a prime ideal of $R$, then  $Ann(Z(R))= \mathrm{Nil}(R)$ and $ann_R(x)\subseteq\mathrm{Nil}(R)$, for every $x\in Z(R)\setminus\mathrm{Nil}(R) $. Since $\mathrm{Nil}(R)$ is a minimal ideal, $ann_R(x)=\mathrm{Nil}(R)$. Hence  $AG(R)=K_{|\mathrm{Nil}(R)^*|} \vee \overline{K_{n}}$, where $n\in\{0, \infty\}$. Therefore, the condition $2\leq|\mathrm{Nil}(R)|\leq3$ implies that $AG(R)$ is planar.}
\end{proof}

\par We are now in a position to classify all finite rings with planar annihilator graphs.
\begin{cor}
 Let $R$  be a finite ring. If $AG(R)$ is planar, then $R$ is isomorphic to one of the following rings:

$\mathbb{Z}_{4}, \,\,\mathbb{Z}_2[x]/(x^2),\, \,\mathbb{Z}_{9},\,\,\mathbb{Z}_3[x]/(x^2), \,\mathbb{Z}_{8},\,\,\mathbb{Z}_2[x]/(x^3), \,\,\mathbb{Z}_4[x]/(x^2-2,2x),\,\, \mathbb{Z}_2[x,y]/(x^2,xy,y^2),$

$\mathbb{Z}_4[x]/(2x,x^2),\,\,\mathbb{F}_4[x]/(x^2),\mathbb{Z}_4[x]/(x^2+x+1),\, \,\mathbb{Z}_{25},\,\,\mathbb{Z}_5[x]/(x^2),\,\mathbb{Z}_2\times\mathbb{F}_{p^n},\,\mathbb{Z}_3\times\mathbb{F}_{p^n},$

$\,\mathbb{Z}_2\times\mathbb{Z}_4,\, \mathbb{Z}_2\times\mathbb{Z}_2[X]/(X^2),\, \mathbb{Z}_2\times \mathbb{Z}_2\times\mathbb{Z}_2.$
\end{cor}
\begin{proof}
{The proof follows from \cite[Section 5]{Redmond}, Theorems \ref{comp5l} and \ref{comp95l}. }
\end{proof}

\par The last result in this section states that the genus of the annihilator graph associated with an infinite ring is either zero or infinite (see the next theorem).
\begin{thm}\label{lemma433}
 Let $R$ be an infinite ring. Then either $\gamma(AG(R))=0$ or $\gamma(AG(R))=\infty$.
\end{thm}
\begin{proof}
{Suppose to the contrary that $0<\gamma(AG(R))<\infty$. We consider two following cases.

\textbf{Case 1.} $R$ is indecomposable. The equality $|R|=\infty$ together with \cite[Theorem 3.10]{Badawi} imply that $Z(R)\neq\mathrm{Nil}(R)$.  Let $x\in Z(R)\setminus\mathrm{Nil}(R)$. Since $R$ is indecomposable, $|Rx|=\infty$, and so $\gamma(AG(R))<\infty$ shows that $|ann_R(x)|\leq 3$. So the indecomposability of $R$ implies that $\mathrm{Nil}(R)\neq (0)$. We claim that  for every $y\in Z(R)\setminus\mathrm{Nil}(R)$, $ann_R(x)=ann_R(y)$. If $ann_R(x)\neq ann_R(y)$ for some  $y\in Z(R)\setminus\mathrm{Nil}(R)$, then since $ann_R(x)$ and $ ann_R(y)$ are  two minimal ideals, $ann_R(x)\cap ann_R(y)=(0)$. Now,  let $ 0\neq a\in ann_R(x)$ and $ 0\neq b\in ann_R(y)$. Since $Ra$ and $Rb$ are two minimal ideals, both $ann_R(a)$ and $ann_R(b)$ are maximal ideals. So we put  $ann_R(a)=\mathfrak{m}_1$ and $ann_R(b)=\mathfrak{m}_2$. We consider two subcases.

\textbf{Subcase 1.} $|\mathfrak{m}_1\cap\mathfrak{m}_2|=\infty$. So  the vertices contained in the set $\{a,b, a+b \}$ and  the vertices contained in the set $\mathfrak{m}_1^*\cap\mathfrak{m}_2^*\setminus \{a,b, a+b \}$ form $K_{3,\infty}$, a contradiction.

\textbf{Subcase 2.} $|\mathfrak{m}_1\cap\mathfrak{m}_2|<\infty$. The indecomposability of $R$ implies that $\mathfrak{m}_1\neq \mathfrak{m}_2$,  $|\mathfrak{m}_1\setminus \mathfrak{m}_2|=\infty$ and  $|\mathfrak{m}_2\setminus \mathfrak{m}_1|=\infty$.
Thus Lemma \ref{lemma43} contradicts $\gamma(AG(R))<\infty$. Hence  for every $y\in Z(R)\setminus\mathrm{Nil}(R)$, $ann_R(x)=ann_R(y)$ and so the claim is proved. This implies that $AG(R)=K_{|ann_R(x)^*|}\vee\overline{K}_\infty$ and so $\gamma(AG(R))=0$, a contradiction.

\textbf{Case 2.} $R$ is decomposable.
 Let $R\cong R_1\times R_2$. Since $0<\gamma(AG(R))<\infty$, we may assume that $|R_1|\leq 3$, $|R_2|=\infty$. Therefore,  $\gamma(AG(R))=0$, a contradiction.~}
\end{proof}
{\section{Toroidal Annihilator  Graphs}\vspace{-2mm}
In this section all rings with toroidal annihilator graphs are classified. We first study annihilator graphs associated with reduced rings.
 \begin{thm}\label{comp5l1}
  Let $R$ be a reduced ring. If  $AG(R)$ is toroidal, then $R\cong R_1\times \cdots \times R_n$, where $2\leq n\leq 3$. Moreover, one of the following statements hold.

  $(1)$ If $n=3$, then  $R\cong \mathbb{Z}_2\times \mathbb{Z}_2\times\mathbb{Z}_3$. Also, $AG(\mathbb{Z}_2\times \mathbb{Z}_2 \times \mathbb{Z}_3)$ is a toroidal graph.

  $(2)$ If $n=2$, then $R$ is one of the following rings:

  $$\mathbb{F}_7 \times \mathbb{F}_4,\,\, \mathbb{F}_5\times \mathbb{F}_5,\,\, \mathbb{F}_5\times \mathbb{F}_4, \,\,\mathbb{F}_4 \times \mathbb{F}_4.$$
\end{thm}
\begin{proof}
{ First we show that $R$ is decomposable.  By hypothesis, $AG(R)$ is a toroidal graph and so it follows from Theorem  \ref{lemma433}  that $R$ is finite. Since $R$ is a reduced ring, we deduce that $R\cong R_1\times \cdots \times R_n$, where $2\leq n$. If $ n\geq 4$, then we prove  that $AG(R)$ is not a toroidal graph. To see this, we need only to check the case $n=4$. If $n=4$, then it is not hard to check that the vertices of the set $\{(1,1,1,0),(1,1,0,0),(1,0,1,0),(0,1,1,0),(1,0,0,0) \}$ and the vertices contained in the set $\{(1,0,0,1),(0,1,0,1),(0,0,1,1),(0,0,0,1) \}$ together with the path $(1,0,0,0)-(0,1,0,0)-(1,0,0,1)$ form a subgraph which  contains a subdivision of $K_{5,4}$, a contradiction. So $n\leq 3$.

$(1)$ Let $R \cong R_1\times R_2 \times R_3$. The ring $R_i$ is indecomposable and finite, for every $1\leq i \leq 3$, so $R_i$ is a field  for every $1 \leq i \leq 3$. If $R_1\cong R_2 \cong  R_3 \cong \mathbb{Z}_2$, then by Theorem \ref{chit4=ome},  $AG(R)$ is a planar graph, a contradiction. So, with no loss of generality, we can suppose that $|R_3|>2$. We show that $R_1 \cong R_2 \cong \mathbb{Z}_2$.   If $|R_2|>2$, then the vertices of the set $\{(1,0,0),(1,0,1),(1,0,y),(1,1,0),(1,x,0) \}$ and the vertices of the set $\{(0,1,1),(0,1,y),(0,x,y),(0,x,1) \}$  form a subgraph which  contains a subdivision of $K_{5,4}$, where $x \in R_2\setminus \{0,1 \}$ and $y \in R_3 \setminus \{0,1\}$, a contradiction. Thus $R_2 \cong \mathbb{Z}_2$. Similarly, $R_1 \cong \mathbb{Z}_2$. We have only to prove that $R_3 \cong \mathbb{Z}_3$ and $AG(\mathbb{Z}_2\times \mathbb{Z}_2 \times \mathbb{Z}_3)$ is a toroidal graph. If $x,y \in R_3 \setminus \{0,1\}$, then the vertices of the set $\{(0,1,x),(0,1,y),(0,1,1),(0,1,0) \}$ and the vertices contained in the set $\{(1,1,0),(1,0,1),(1,0,x),(1,0,y),(1,0,0) \}$ together with the path $(0,1,0)-(0,0,1)-(1,1,0)$  form a subgraph which  contains a subdivision of $K_{5,4}$, a contradiction. Hence $R_3 \cong \mathbb{Z}_3$ and  $R \cong \mathbb{Z}_2\times \mathbb{Z}_2 \times \mathbb{Z}_3$. The following Figure shows that $AG(\mathbb{Z}_2\times \mathbb{Z}_2 \times \mathbb{Z}_3)$ can, indeed, be drawn without crossing itself on a torus.  Hence  $AG(\mathbb{Z}_2\times \mathbb{Z}_2 \times \mathbb{Z}_3)$ is a toroidal graph.

\begin{center}
\begin{figure}[h]
\hspace{4.2cm}
\unitlength=.6mm
\begin{picture}(100,90)(-70,-20)
\put (-70,-10){\framebox(120,80)}
\put (-15,50){\circle*{1.5}}
\put (-13,52){\tiny{(1,1,0)}}
\put (-15,50){\line (1,0){65}}
\put (-15,50){\line (2,-1){30}}
\put (-25,35){\circle*{1.5}}
\put (-23,36){\tiny{(1,0,1)}}
\put (15,35){\line (1,0){35}}
\put (-5,35){\circle*{1.5}}
\put (-4,35){\tiny{(0,1,1)}}
\put (15,25){\line (3,-1){35}}
\put (-5,35){\line (-2,3){10}}
\put (-5,35){\line (-2,0){20}}
\put (-5,35){\line (2,-1){20}}
\put (-5,35){\line (0,-1){20}}
\put (-25,35){\line (2,3){10}}
\put (-25,35){\line (-3,5){21}}
\put (-25,15){\circle*{1.5}}
\put (-31,17){\tiny{(0,1,2)}}
\put (-25,15){\line (-4,-1){45}}
\put (-5,15){\circle*{1.5}}
\put (-4,15){\tiny{(1,0,2)}}
\put (-5,15){\line (-2,0){20}}
\put (-35,35){\circle*{1.5}}
\put (-39,37){\tiny{(0,1,0)}}
\put (-35,35){\line (3,-2){30}}
\put (-35,35){\line (-5,-3){35}}
\put (-35,35){\line (-5,2){35}}
\put (-35,35){\line (-1,0){35}}
\put (-35,35){\line (1,0){15}}
\put (15,25){\circle*{1.5}}
\put (17,26){\tiny{(1,0,0)}}
\put (15,25){\line (5,-3){35}}
\put (15,35){\circle*{1.5}}
\put (17,37){\tiny{(0,0,2)}}
\put (15,35){\line (0,-1){10}}
\put (15,50){\circle*{1.5}}
\put (17,52){\tiny{(0,0,1)}}
\put (-25,15){\line (-4,-5){20}}
\put (-15,50){\line (-1,2){10}}
\put (-15,50){\line (1,2){10}}
\put (-25,15){\line (0,-1){25}}
\put (-5,15){\line (0,-1){25}}
\put (15,50){\line (0,1){20}}
\put (15,25){\line (0,-1){35}}
\end{picture}
\caption{$ \text{The} \,\, \text{annihilator} \,\, \text{graph }\, \text{of}\,\, \mathbb{Z}_2\times  \mathbb{Z}_2\times \mathbb{Z}_3 \,\, \text{on} \text{ the} \text{ torus.} $}
\end{figure}
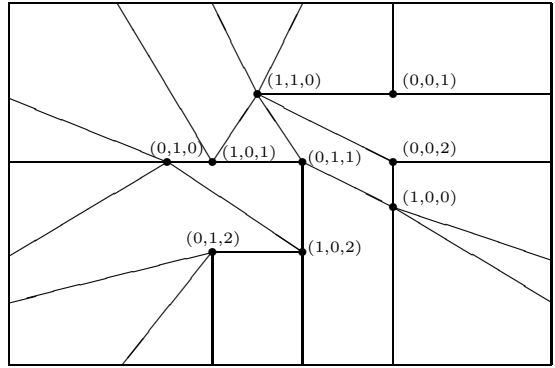

\end{center}
 $(3)$ If $n=2$, then the result follows from  \cite[Theorem 3.6]{Badawi}, and part $(3)$ of  \cite[Theorem 3.1]{wickham}.  }
\end{proof}

\par To complete our classification, we state the following remark and lemma.
 \begin{remark}\label{remar}
  It is not hard to see that, if  $(R,\mathfrak{m}) $  is a finite local  ring, then there exists a prime integer $p$ and positive integers $t, l, k$ such that \rm{char}$(R)=p^t$,  $|\mathfrak{m}|=p^l$, $|R|=p^k$ and \rm{ char}$(R/ \mathfrak{m})=p$.
\end{remark}

\begin{lem}\label{list ring}
 Let $(R,\mathfrak{m})$  be a finite local ring. If $|\mathfrak{m}|\in \{7,8\}$, then $R$ is isomorphic to one of the following 22 rings:

$\mathbb{Z}_{49}, \,\,\mathbb{Z}_7[x]/(x^2),\, \,\mathbb{Z}_{16},\,\,\mathbb{Z}_2[x]/(x^4), \,\,\mathbb{Z}_4[x]/(x^2+2), \,\,\mathbb{Z}_4[x]/(x^2+3x),\,\,\mathbb{Z}_4[x]/(x^3-2,2x^2,2x),\,\,$

$\mathbb{Z}_2[x,y]/(x^3,xy,y^2),\,\mathbb{Z}_8[x]/(2x,x^2),\,\mathbb{Z}_4[x]/(x^3,2x^2,2x),\,\mathbb{Z}_4[x]/(x^2+2x),\,\mathbb{Z}_8[x]/(2x,x^2+4),\,$
$\mathbb{Z}_2[x,y]/(x^2,y^2-xy),\,\,\mathbb{Z}_4[x,y]/(x^2,y^2-xy,xy-2,2x,2y),\,\, \mathbb{Z}_4[x,y]/(x^3,y^2,xy-2,2x,2y),\,\, $

$ \mathbb{Z}_2[x,y]/(x^2,y^2), \,\,\mathbb{Z}_4[x]/(x^2), \,\mathbb{Z}_4[x]/(x^3-x^2-2,2x^2,2x),\,\,\mathbb{Z}_2[x,y,z]/(x,y,z)^2, \,\,\mathbb{F}_8[x]/(x^2),\,\,$

$\mathbb{Z}_4[x,y]/(x^2,y^2,xy,2x,2y),\,\, \mathbb{Z}_4[x]/(x^3+x+1)$
\end{lem}
\begin{proof}
{The proof follows from \cite[Section 5]{Redmond}.}
\end{proof}
\par We are now in a position to classify toroidal annihilator graphs associated with  non-reduced ring.
 \begin{thm}\label{nonreduced}
  Let $R$ be a non-reduced ring. If  $AG(R)$ is  toroidal, then $R\cong R_1\times \cdots \times R_n$, where $ n\leq 2$. Moreover, one of the following statements hold.

$(1)$ If $n=1$, then $R$ is one of the following rings:

$\mathbb{Z}_{49}, \,\,\mathbb{Z}_7[x]/(x^2),\, \,\mathbb{Z}_{16},\,\,\mathbb{Z}_2[x]/(x^4), \,\,\mathbb{Z}_4[x]/(x^2+2), \,\,\mathbb{Z}_4[x]/(x^2+3x),\,\,\mathbb{Z}_4[x]/(x^3-2,2x^2,2x),\,\,$

$\mathbb{Z}_2[x,y]/(x^3,xy,y^2),\,\mathbb{Z}_8[x]/(2x,x^2),\,\mathbb{Z}_4[x]/(x^3,2x^2,2x),\,\mathbb{Z}_4[x]/(x^2+2x),\,\mathbb{Z}_8[x]/(2x,x^2+4),\,$
$\mathbb{Z}_2[x,y]/(x^2,y^2-xy),\,\,\mathbb{Z}_4[x,y]/(x^2,y^2-xy,xy-2,2x,2y),\,\, \mathbb{Z}_4[x,y]/(x^3,y^2,xy-2,2x,2y),\,\, $

$ \mathbb{Z}_2[x,y]/(x^2,y^2), \,\,\mathbb{Z}_4[x]/(x^2), \,\mathbb{Z}_4[x]/(x^3-x^2-2,2x^2,2x),\,\,\mathbb{Z}_2[x,y,z]/(x,y,z)^2, \,\,\mathbb{F}_8[x]/(x^2),\,\,$

$\mathbb{Z}_4[x,y]/(x^2,y^2,xy,2x,2y),\,\, \mathbb{Z}_4[x]/(x^3+x+1)$

  $(2)$ If $n=2$, then either $ R\cong \mathbb{Z}_4 \times \mathbb{Z}_3$ or $R\cong \mathbb{Z}_2[x]/(x^2) \times \mathbb{Z}_3$.
\end{thm}

\begin{proof}
{ By Theorem   \ref{lemma433}, $R$ is finite and so is an Artinian ring. Thus $R\cong R_1\times \cdots \times R_n$, where $ n\geq 1$. Let $R\cong R_1\times \cdots \times R_n$, where $ n\geq 3$. Since $R$ is a non-reduced ring, we can suppose that  $\mathrm{Nil}(R_1)\neq(0)$. This implies that $|U(R_1)| \geq 2$. Let $a\in \mathrm{Nil}(R_1)^*$  and $1\neq u\in U(R_1)$.  We show that $AG(R)$ is not a toroidal graph. To see this, we need only to check the case $n=3$. But if $n=3$, then by  Lemma \ref{lemma2}, one may see that the vertices of the set $\{(1,0,0),(u,0,0),(1,1,0),(u,1,0),(a,1,0) \}$ and the vertices contained in the set $\{(0,1,1),(0,0,1),(a,0,1),(a,1,1) \}$ form a subgraph which contains a subdivision of $K_{5,4}$, a contradiction. So $n\leq 2$.

$(1)$ Let $(R,\mathfrak{m})$  be a local ring.  By Theorem   \ref{lemma433}, $R$ is finite. So  $|R|=p^k$ and $|\mathfrak{m}|=p^l$, for some prime number  $p$ and some integers $k,l$. If $| \mathfrak{m}|>8 $, then by \cite[Theorem 3.10]{Badawi}, $AG(R)$ is a not  toroidal graph. Thus  $|\mathfrak{m}|\leq 8$. Since  $|\mathfrak{m}|\geq 6$ and $|\mathfrak{m}|=p^l$, for some prime number $p$ and for some integer $l$, we deduce that either $|\mathfrak{m}|= 8$ or $|\mathfrak{m}|= 7$. Thus by Lemma \ref{list ring}, the result holds.

 $(2)$ Suppose that $R\cong R_1\times R_2$, where $(R_i,\mathfrak{m}_i )$ is a finite local ring, for $1\leq i \leq 2$.  With no loss of generality, suppose that  $\mathrm{Nil}(R_1)\neq(0)$. First, we show that $|\mathfrak{m}_1|=2$. If $|\mathfrak{m}_1|>2$, then $|R_1|\geq 9$. So the vertices of the set $\{(1,0),(a_1,0),(a_2,0),(a_3,0),(a_4,0),(a_5,0),(a_6,0)\}$ and vertices contained in the $\{(0,1),(0,1),(a,1),(b,1) \}$ form a subgraph which  contains a subdivision of $K_{7,4}$ where $a,b \in\mathrm{Nil}(R_1)^* $ and $a_i\in R_1\setminus \{0,1\}$ for $1\leq i \leq 6$, a contradiction. Hence $|\mathfrak{m}_1|=2$. Thus either $R_1=\mathbb{Z}_4 $ or $R_1=\mathbb{Z}_2[x]/(x^2)$. Next, we show that $R_2$ is a field. To see this,  let $a\in \mathfrak{m}_2^*$ and $R_1=\mathbb{Z}_4 $. Then  by Lemma \ref{lemma2}, the vertices contained in two sets $\{(2,a),(2,1),(0,1),(0,a) \}$ and $\{(1,a),(3,a),(1,0),(3,0),(2,0) \}$ form a subgraph which  contains a subdivision of $K_{5,4}$, a contradiction. Therefore, $R_2$ is a field. If $|R_2|\geq 5$ and $R_1=\mathbb{Z}_4 $, then the vertices contained in two  sets $$\{(2,1),(2,a_1),(2,a_2),(2,a_3),(0,1),(0,a_1),(0,a_2),(0,a_3) \}$$ and $\{(1,0),(2,0),(3,0) \}$ form a subgraph which  contains a subdivision of $K_{8,3}$, a contradiction. This implies that $R_2=\mathbb{Z}_2$, $R_2=\mathbb{Z}_3$ or $R_2=\mathbb{F}_4$. If $R_2=\mathbb{Z}_2$, then  by \cite[Theorem 3.16]{Badawi}, $AG(R)=K_{2,3}$ and so  $AG(R)$ is not a toroidal graph. If $R_2=\mathbb{Z}_3$, then we can easily check that  $AG(R)$ contains  $K_{3,3}$ as a subgraph and since in this case $|V(AG(R))|=7$, we conclude that $AG(R)$ is a toroidal graph. Hence if $R_2= \mathbb{Z}_3 $, then there are two rings that $AG(R)$ is a toroidal graph. They are: $ \mathbb{Z}_4 \times \mathbb{Z}_3,\mathbb{Z}_2[x]/(x^2) \times \mathbb{Z}_3 $.

If $R_2=\mathbb{F}_4$, then $R=\mathbb{Z}_4\times \mathbb{F}_4$. Assume that $\mathbb{F}_4=\{0, u_1, u_2, u_3\}$. Let $x=(1,0), y=(2,0), z=(3,0), a=(0,u_1), b=(0,u_2), c=(0,u_3), d=(2,u_1), e=(2,u_2), f=(2,u_3), V_1=\{x,y,z \}$ and $V_2=\{a,b,c,d,e,f\} $. It is not hard to check that $AG(R)$ is $K_{|V_1|,|V_2|}$  together with a triangle in $V_2$. Therefore, $AG(R)$ is not a toroidal graph and so the proof is complete.}
\end{proof}
{}


\begin{thebibliography}{}{\small

\bibitem{khashyarmanesh} M. Afkhami, M. Karimi, K. Khashyarmanesh, On the regular digraph of ideals of  commutative rings, Bull. Aust. Math. Soc, 88 (2013) 177--189.

\bibitem{nikak} S. Akbari, R. Nikandish, Some results on the intersection graphs of ideals of matrix algebras, Linear and Multilinear Algebra, 62 (2) (2014) 195--206.

\bibitem{nikandish} A. Alibemani, M. Bakhtyiari, R. Nikandish, M.J. Nikmehr, The annihilator ideal graph of a commutative ring, J. Korean Math. Soc, 52 (2) (2015)  417�-429.

\bibitem{anderson} D.F. Anderson, P.S. Livingston, The zero-divisor graph of a commutative ring,
 J. Algebra, 217 (1999) 434--447.

\bibitem{Naseer} D.D. Anderson, M. Nassr, Beck's coloring of  a commutative Ring,
 J. Algebra, 159 (1993) 500--514.

\bibitem{Badawi} A. Badawi, On the annihilator graph of a commutative ring,
 Comm. Algebra, 42 (2014) 108�-121.

\bibitem{ab} A. Badawi, On the dot product graph of a commutative ring, Comm.
Algebra, 43  (2015) 43--50.

\bibitem{Beck} I. Beck, Coloring of commutative rings, J. Algebra, 116 (1988) 208--226.

\bibitem {her} W. Bruns, J. Herzog, Cohen-Macaulay Rings, Cambridge University Press, 1997.

\bibitem{Corbas} B. Corbas, G.D, Rings with few zero divisors, Math. Ann, 181 (1969)  1--7.

\bibitem{Corbas1} B. Corbas, G.D. Williams, Rings of order $p^5$ part I. Nonlocal rings,  J. Algebra, 231 (2000), 677--690.



\bibitem{ColorBadawi}   R. Nikandish, M.J. Nikmehr, M. Bakhtyiari, Coloring of the cnnihilator graph of a commutative ring,
 J. Alg. Appl, to appear.

\bibitem{Heydari} M.J. Nikmehr,  F. Heydari, The M-principal graph of a commutative ring, Period. Math. Hung, 68 (2014) 185--192.

\bibitem{Redmond} S. P. Redmond, On zero-divisor graphs of small finite commutative rings, Discrete  Mathematics, 307 (2007) 1155--1166.

\bibitem {sha} R.Y. Sharp, Steps in Commutative Algebra, second edition, London Mathematical Society Student Texts 51, Cambridge University Press, Cambridge, 2000.

\bibitem {su} H. Su, G. Tang, Y. Zhou, Rings whose unit graphs are planar, Publ Math-Debrecen, 86 (2015) 3-4 (8).

\bibitem{west} D.B. West, Introduction to Graph Theory, 2nd ed., Prentice Hall, Upper Saddle River (2001).

\bibitem{wickham} C. Wickham,  Classification of rings with genus one zero-divisor graphs, Comm. Algebra, 36 (2008) 325--345.

\bibitem{Rob} R. Wisbauer,  Foundations of Module and Ring Theory, Breach Science Publishers 1991.

}
\end{thebibliography}
\end{document}